\documentclass[12pt,a4paper]{amsart}
\usepackage{amsfonts}
\usepackage{amsthm}
\usepackage{amsmath}
\usepackage{amscd}
\usepackage[latin2]{inputenc}
\usepackage{t1enc}
\usepackage[mathscr]{eucal}
\usepackage{indentfirst}
\usepackage{graphicx}
\usepackage{graphics}
\usepackage{pict2e}
\usepackage{epic}
\usepackage{enumitem}
\numberwithin{equation}{section}
\usepackage[margin=1.5in]{geometry}
\usepackage{epstopdf} 
\usepackage{float}
\usepackage{setspace}
\usepackage{bbm}
\theoremstyle{plain}
\newtheorem{theorem}{Theorem}[section]
\newtheorem{lemma}[theorem]{Lemma}
\newtheorem{corollary}[theorem]{Corollary}
\newtheorem{proposition}[theorem]{Proposition}
\theoremstyle{definition}
\newtheorem{definition}{Definition}[section]

\newtheorem{example}{Example}[section]
\newtheorem{innercustomthm}{Theorem}
\newenvironment{customthm}[1]
  {\renewcommand\theinnercustomthm{#1}\innercustomthm}
  {\endinnercustomthm}

\newcommand{\Q}{\mathbb{Q}}
\newcommand{\Z}{\mathbb{Z}}
\newcommand{\R}{\mathbb{R}}
\newcommand{\1}{\mathbbm{1}}
\newcommand{\Hom}{\text{Hom}}
\newcommand{\Ext}{\text{Ext}}
\newcommand{\Tor}{\text{Tor}}

\DeclareMathOperator*{\Conf}{Conf}
\DeclareMathOperator*{\Ker}{Ker}

\newcommand{\og}{\mathscr{O}_G}
\newcommand{\cg}{\mathscr{C}_G}

\usepackage{pgfplots}

\begin{document}

\title{On the Rational Bredon Cohomology of Equivariant Configuration Spaces}
\author[Q.Zhu]{Qiaofeng Zhu}
\address{University of Massachusetts, Boston \\ Department of Mathematics \\
Roston, MA, 02125} 
\email{Qiaofeng.Zhu@umb.edu}


\begin{abstract} 
Bredon cohomology is a cohomology theory that applies to topological spaces
equipped with the group actions. For any group $G$, given a real linear representation $V$ , the configuration space of $V$ has a natural diagonal $G$-action. In the paper we study this group action on the configuration space and give a decomposition of the homology Bredon coefficient system of the configuration space and apply this to compute Bredon cohomology of the configuration space for small non-abelian group $G$.
 
\end{abstract}

\maketitle

\setcounter{tocdepth}{2}
\tableofcontents
\section{Introduction}
The objects of study in equivariant algebraic topology are spaces equipped with an action by a topological group $G$.
The category of Bredon coefficient systems over a commutative ring $R$, denote by $\cg^R$, is the category of contravariant functors from the canonical orbit category of $G$ to the category of $R$-modules. It is an Abelian category with enough injectives. Using this category $\cg^R$, Bredon defined a homology and a cohomology theory for $G$-spaces \cite{Bredon67}.  

Through this paper we will assume that $G$ is a finite group. Given a $G$-space $X$, in order to compute Bredon cohomology, we need to access the homology of the fixed points sets $X^H$ for all subgroups $H$ of $G$ as well as an injective resolution for an arbitrary Bredon coefficient system. 

For an arbitrary $G$-space $X$, it might be a too difficult task to give a closed formula for the fixed point sets of all subgroups $H<G$. However, for particular types of spaces, we can obtain such formulas, which makes computing Bredon cohomology possible. In \cite{Qiaofeng18}, the author computed the fixed point sets for polyhedral products. In this paper, the objects to study is equivariant configuration spaces.

\begin{definition}
Let $G$ be a finite group. For an $\R$-linear representation $V$, the equivariant configuration space of $V$ is the space $\Conf(V,q)$ with the diagonal $G$-action.
\end{definition}
We prove the following fixed point theorem in Section 7.
\begin{customthm}{\ref{fixed point set of ecs}}
For any subgroup $H<G$, the H-fixed point $\Conf(V,q)^H$ is $\Conf(V^H,q)$.
\end{customthm}

Define the Weyl group $WH$ for any subgroup $H$ of $G$ by $$WH=N_G(H)/H.$$

\begin{definition}
For any subgroup $H<G$, let $V_H$ be a left $\Q(WH)$-module. Define a Bredon coefficient system $I(V_H)$ by 
$$\underline{I(V_H)}(G/K)= \Hom_{\Q(WH)}(\Q((G/K)^H), V_H).$$ 
For a $G$-map $f:G/K\rightarrow G/K'$, the map $\overline{f}: \Q((G/K)^H)\rightarrow \Q((G/K')^H)$ is induced by $f$. Then 
$$\underline{I(V_H)}(f): \underline{I(V_H)}(G/K')\rightarrow \underline{I(V_H)}(G/K)$$ is defined by $\underline{I(V_H)}(f)(g)=g\circ \overline{f}$ where $g\in I(V_H)(G/K')$. 
\end{definition}

In \cite{Doman88}, Doman proved Bredon coefficient systems in the form $\underline{I(V_H)}$ are injective. In addition, he provide an injective envelope for any Bredon coefficent system over $\Q$ using those injective coefficient systems.

\begin{theorem}[\cite{Doman88}]
$f:\underline{M}\rightarrow\oplus \underline{I(V_H)}$ is a injective envelope of $\underline{M}$, where direct sum is over all conjugacy classes of $G$.
\end{theorem} 

One of the tools for computing Bredon cohomology is the universal coefficient spectral sequences.

\begin{theorem}[\cite{Bredon67}, \cite{May96}]
There is a universal coefficient spectral sequence that converges to Bredon cohomology 
$$E_2^{p,q}=\Ext_{\cg^R}^{p,q}(\underline{H_*}(X), \underline{M})\Rightarrow H_G^{p+q}(X, \underline{M}),$$
and a universal coefficient spectral sequence that converges to Bredon homology 
$$E^2_{p,q}=\Tor^{\cg^R}_{p,q}(\underline{H_*}(X), \underline{N})\Rightarrow H^G_{p+q}(X, \underline{M}).$$
\end{theorem}

For equivariant spaces, the homology Bredon coefficient system has a decomposition. 
\begin{definition}
The Bredon coefficient system $\underline{\1_H}$ over $\Q$ is given by 
\[
\underline{\1_H}(G/K)=\left\{
\begin{array}{ll}
\Q,& \text{if $K$ conjugates to $H$,}\\
0, & \text{otherwise. }
\end{array}\right.
\]
\end{definition}

In section 7, we have the following result for the Bredon cohomology of equivariant configuration spaces.
\begin{customthm}{\ref{decompostion of homology bcs}}
For a finite group $G$, $V=\R[G]$ is the regular $\R$-linear representation of $G$. The following is a decomposition of the rational homology Bredon coefficient system of $\Conf(V,q)$,
\begin{enumerate}[label=\roman*)]
\item
For $n >0$, $$\underline{H_n}(\Conf(V,q))=\bigoplus_{H<G}\beta_{H,n}\underline{\1_H}$$ where
$\beta_{H,n}$ is the $n$-th Betti number of $\Conf(V^H,q)$.
\item
For $n=0$, $$\underline{H_0}(\Conf(V,q))=\underline{\Q}\oplus(q!-1)\underline{\1_{\{e\}}}$$
where $\underline{\Q}$ is the constant $\Q$ coefficient system.
\end{enumerate}
\end{customthm}

This decompostion allows us to compute Bredon cohomology of equivariant configuration spaces. In chapter 5, we compute the universal coefficient spectral sequence for the Bredon cohomology of $\Conf(\R[G],3)$ and $\Conf(\R[G],4)$ with the coefficient system $\1_0$.

\section{Bredon coefficient systems}
\begin{definition}
The canonical orbit category of group $G$, denoted by $\og$, is the category whose objects are G-spaces $G/H$ and morphisms are $G$-maps. 
\end{definition}

There is a $G$-map $f:G/H \rightarrow G/K$ if and only if $gHg^{-1}<K$. Notice that if $f(eH)=gK$ for some $G\in G$, then $$eK=f(g^{-1}H)=f(g^{-1}hg\cdot g^{-1}H)= g^{-1}hg\cdot f(g^{-1}H)=g^{-1}hg \cdot eK$$ for any $h\in H$. Hence $g^{-1}Hg<K$.

\begin{definition}
Let $R$ be a commutative ring and $Mod_R$ be the category of $R$-modules. A Bredon coefficient system over $R$ is a contravariant functor $\og\rightarrow Mod_R$. The category of Bredon coefficient systems over $R$ is denoted by $\cg^R$. When $R=\mathbb{Z}$, we  use $\cg$ for simplicity.
\end{definition}

\begin{example}
Given a based $G$-CW-complex, the n-th equivariant homotopy group $\underline{\pi_n}(X)$ for $n\geq 2$, is a Bredon coefficient system given by $$\underline{\pi_n}(X)(G/H)=\pi_n(X^H)$$
\end{example}

\begin{example}
Let X be a $G$-CW-complex, we can define the cellular chain complex of coefficient systems $\underline{C_*}(X)$, where
$$\underline{C_n}(X)(G/H)=H_n((X^n)^H, (X^{n-1})^H, R).$$
Define $$\underline{H_n}(X)=H_n(\underline{C_*}(X)).$$
\end{example}

By a general categorical argument, the category $\cg^R$ is an Abelian category with enough injectives. So we could talk about homological algebra concept such as homology and cohomology of chains and cochains in this category.

\begin{definition}
\label{Bredon chomology and homology}
Let $\underline{M}$ be a Bredon coefficient system and $X$ be a $G$-CW-complex. Define a cochain complex
$$\underline{C^n}(X; \underline{M})=\Hom_{\cg^R}(\underline{C_n}(X),\underline{M}),$$
Its cohomology, $$H^*_G(X, \underline{M}):=H^*(\underline{C^*}(X;\underline{M}))$$ is called the Bredon cohomology of $X$ with coefficient $\underline{M}$. 
\end{definition}

To define Bredon homology, we need a covariant functor $\underline{N}:\og\rightarrow Mod_R$ as a coefficient system. Define the cellular chain by 
$$C_n(X; \underline{N})=\underline{C_n}(X)\otimes_{\cg^R} \underline{N}=\int^{G/H}\underline{C_n}(X)(G/H)\otimes_R \underline{N}(G/H)$$
In other word, the tensor product $\otimes_{\cg^R}$ is given by the coend of the two functors. 

\begin{definition}
Bredon homology of $X$ with coefficient $N$ is given by 
$$H^G_n(X,\underline{N})=H_n(C_*(X; \underline{N}))$$
\end{definition}

\begin{theorem}[\cite{Bredon67},\cite{May96}]
\label{universal coefficient spectral sequence}
There are universal coefficient spectral sequences
$$E_2^{p,q}=\Ext_{\cg^R}^{p,q}(\underline{H_*}(X), \underline{M})\Rightarrow H_G^n(X, \underline{M}),$$
and
$$E^2_{p,q}=\Tor^{\cg^R}_{p,q}(\underline{H_*}(X), \underline{N})\Rightarrow H^G_n(X, \underline{M}).$$
\end{theorem}

\section{Bredon coefficient systems in the language of path algebras}
Using the language of path algebra, we could also give a more ``algebraic'' description of Bredon cohomology and homology. In this subsection, we will define path algebra and give an alternative definition of Bredon coefficient systems as well as Bredon homology and cohomology. The readers who are interested in this topic could check for more details of path algebras and the related representation theory in the summary book \cite{ASS06}.

\begin{definition}
A quiver $Q=(V, E, s, t)$ is a directed graph, i.e. a graph that associates each edge a direction, where $V$ is the set of vertices and $E$ is the set of edges along with two maps $s,t: E\rightarrow V$ that for each edge $\alpha\in E$, the images $s(\alpha)$ and $t(\alpha)$ are the source and target of the edge $\alpha$ respectively. 
\end{definition}

\begin{example}
A small category $\mathscr{C}$ is naturally a quiver whose objects and morphisms are vertices and edges of the quiver respectively. We denote the quiver associated  with $\mathscr{C}$ by $Q_\mathscr{C}$. 
\end{example}

\begin{definition}
For any quiver $Q=(V, E, s, t)$, the set of paths of $Q$, denoted by  $P_Q$, consists of the following elements:
\begin{enumerate}[label=\roman*)]
\item
For each vertex $v\in V$, there is a trivial constant path $e_v$, and set $s(e_v)=t(e_v)=v$;
\item
All the finite sequences $\alpha_n\alpha_{n-1}\cdots\alpha_1$ where $\alpha_i\in E$ for each $i$ and $t(\alpha_k)=s(\alpha_{k+1})$ for $k=1,2,\cdots, n-1$. In other word, an actual path on the quiver.
\end{enumerate}
\end{definition}

Moreover, we can define a multiplication $\circ$ on $P_Q$. For any two paths $\alpha_n\alpha_{n-1}\cdots\alpha_1$, $\beta_m\beta_{m-1}\cdots\beta_1 $ in $P_Q$,
 
\[
\alpha_n\cdots\alpha_1 \circ \beta_m\cdots\beta_1 =\left\{
\begin{array}{cl}
\alpha_n\cdots\alpha_1\beta_m\cdots\beta_1 & \text{if } t(\beta_m)=s(\alpha_1),\\
0 & \text{otherwise.}
\end{array}
\right.
\]

\begin{definition}
Let $R$ be a commutative ring, the path algebra $RQ$ is an associative algebra with basis $P_Q$ whose multiplication is linearly induced by the multiplication $\circ$ on $P_Q$ over $R$. In addition, if $V$ is an finite set, then the algebra $RQ$ has a multiplicative identity $$1_{RQ}=\sum_{v\in V}e_v,$$ and these $e_v$'s are nonzero idempotents of $RQ$.
\end{definition}
\begin{example}
Let $Q$ be the quiver given by Figure \ref{fig:An}. Then over a field $k$, the path algebra $kQ$ is a $k$-algebra of dimension $\frac{n^2+n}{2}$. And it is isomorphic to the algebra of $n\times n$ upper triangular matrices over $k$.

\begin{figure}[H]
\centering
\begin{tikzpicture}
\filldraw [black] (-5,0) circle (2pt) node[above] {1};
\filldraw [black] (-3,0) circle (2pt) node[above] {2};
\filldraw [black] (-1,0) circle (2pt) node[above] {3};
\filldraw [black] (4,0) circle (2pt) node[above] {n-1};
\filldraw [black] (6,0) circle (2pt) node[above] {n};
\filldraw [black] (1.4,0) circle (1pt);
\filldraw [black] (1.8,0) circle (1pt);
\filldraw [black] (2.2,0) circle (1pt);
\draw[ultra thick, ->](-4.8,0)--(-3.2, 0);
\draw[ultra thick, ->](-2.8,0)--(-1.2, 0);
\draw[ultra thick, ->](-0.8,0)--(0.8, 0);
\draw[ultra thick, ->](2.8,0)--(3.8, 0);
\draw[ultra thick, ->](4.2,0)--(5.8, 0);
\end{tikzpicture}
\caption{Dynkin diagram of type $A_n$} 
\label{fig:An}
\end{figure}
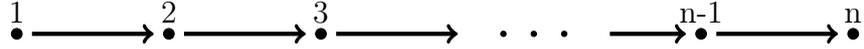
\end{example}

\begin{definition}
A relation of a quiver $Q$ is a subspace of RQ spanned by linear combinations of paths having a common source and a common target. Let $S$ be a set of relations  of $Q$. The path algebra with relation $S$, denoted by $RQ_S$ is $RQ/I_S$ where $I_S$ is the two-sided ideal generated by $S$. 
\end{definition}

For an algebra $A$ over $R$, Let $Mod_{A}$ be the category of right-$A$-modules and ${}_{A}Mod$ be the category of left-$A$-modules.

\begin{definition}
If $G$ is a finite group, the canonical orbit category $\og$ is a small category with finitely many objects. Let $Q= Q_{\og}$ be the quiver associated with $\og$. And $S$ is the set of relations given by the equivalences of morphisms in the category $\og$. $RQ_S$ is the corresponding path algebra.
\end{definition}

\begin{lemma}
\label{in the path algebra language lemma}
There is an equivalence of abelian categories $\cg^R\rightarrow Mod_{RQ_S}$. Then a Bredon coefficient system could be treated as a right-$RQ_S$-module and vice versa. 
\end{lemma}

\begin{proof}
We define two functors $F: \cg^R\rightarrow Mod_{RQ_S}$ and $G: Mod_{RQ_S}\rightarrow\cg^R$ as follows:
\begin{enumerate}[label=\roman*)]
\item
For any Bredon coefficient system $\underline{M}\in \cg^R$, let $$F(M):=\bigoplus_{H<G}\underline{M}(G/H).$$ The path algebra $RQ_S$ action on $F(M)$ is induced by structure maps of $\underline{M}$. Since the Bredon coefficient system is a contravariant funtor $\og\rightarrow Mod_R$, the R-linear space $F(M)$ admits a right $RQ_S$-module structure. 
\item
For any object $G/H$ in $\og$, we have the trivial path $e_{G/H}$. It is also an idempotent element of the path algebra $RQ_S$. For any right $RQ_S$-module $N$ in $Mod_{RQ_S}$, define a Bredon coefficient system $G(N)(G/H)=N.e_{G/H}$.
\end{enumerate}
The two functors $F$ and $G$ give the natural equivalence $$F: \cg^R\longleftrightarrow Mod_{RQ_S}:G.$$ 
\end{proof}

\begin{example}
Let $G=C_2$, the cyclic group of order 2.  the canonical orbit category is shown in Figure \ref{fig:C2}. The path algebra with relations corresponds to $\og$ has dimension 5, and a basis is given by $\{e_{C_2/C_2}, e_{C_2/0}, \alpha, \beta, \beta\alpha\}$. Notice that the loop $\beta$ is given by the Weyl group action. And since the Weyl group for trivial group is the whole group $C_2$, then set of relations in this case is $\{\beta^2\}$.

\begin{figure}[H]
\centering
\begin{tikzpicture}[scale=0.7]
\filldraw[black] (0,0) circle (2pt) node[above left] {$C_2/0$};
\filldraw[black] (0,3) circle (2pt) node[above] {$C_2/C_2$};
\node[left] at (0,1.5) {$\alpha$};
\node at (0,-1.5) {$\beta$};
\draw[ultra thick, ->](0,0.2)--(0,2.8);
 \draw [ultra thick, ->] (-0.3,0) arc(105:435:1);
\end{tikzpicture}
\caption{Canonical orbit category of $C_2$} 
\label{fig:C2}
\end{figure}
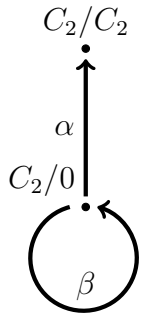
\end{example}

Similarly, a covariant functor $\underline{N}:\og\rightarrow Mod_R$ is equivalent to a left-$RQ_S$-module and the coend of a contravariant functor with a covariant functor is the tensor product of a right $RQ_S$-module with a left $RQ_S$-module over $RQ_S$. We could restate the definition of Bredon homology using the path algebra terminology.

\begin{definition}
Let $X$ be a $G$-CW-complex and $N$ be a left $RQ_S$-module. The cellular chain complex of coefficient systems $\underline{C_*}(X)$, where
$$\underline{C_n}(X)(G/H)=H_n((X^n)^H, (X^{n-1})^H, R).$$ Apply Lemma \ref{in the path algebra language lemma}, $F(\underline{C_n}(X))$ is a right $RQ_S$-module. Define the chain of $R$-module $C_*(X;N)$ by 
$$C_n(X; N)=F(\underline{C_n}(X))\bigotimes_{RQ_S} N$$

The Bredon homology of $X$ with coefficient system $N$ is given by 
$$H^G_n(X,N)=H_n(C_*(X; N)).$$
\end{definition}

\section{Reduced Bredon coefficient systems}
The computational complexity of Bredon cohomology is strongly related the complexity of the canonical orbit category $\og$. In this section, we reduce the Bredon coefficient system to a simpler form. This reduction is quite useful when we compute Bredon cohomology.

 Since the conjugation of subgroups in $G$ induces isomorphisms in the canonical orbit category $\og$, it is enough that we just consider one subgroup of $G$ for each conjugacy class to obtain a category which is equivalent to $\og$ but with fewer objects. Let $\widehat{\og}$ be the full subcategory of $\og$ whose object set consists of one and only one object for each isomorphism class in $\og$. This definition is not canonical, but $\widehat{\og}$ is unique up to category equivalence.
 
\begin{definition}
A reduced Bredon coefficient system over a commutative ring $R$ is a contravariant functor $$\underline{M}: \widehat{\og}\rightarrow Mod_R.$$ We denote the category of reduced Bredon coefficient systems by $\widehat{\cg^R}$.
\end{definition}
 
\begin{lemma}
The category $\widehat{\cg^R}$ is equivalent to $\cg^R$.
\end{lemma}

\begin{proof}
The two functor categories $\widehat{\cg^R}$ and $\cg^R$ have the same target category $Mod_R$. Their source categories $\widehat{\og}$ and $\og$ are equivalent. Hence the functor categories are equivalent as well. The equivalence is induced by the equivalence between $\widehat{\og}$ and $\og$.
\end{proof}
Therefore in the future discussion, we will only talk about reduced Bredon coefficient systems and the reduced canonical orbit category. In order to simplify the notation, denote the reduced canonical orbit category again by $\og$.

\begin{example}
Let $G=\Sigma_3$, the (reduced) canonical orbit category is shown in Figure \ref{fig:S3}. Notice that $\sigma,\tau$ are the generators of the Weyl group $W\{e\}=\Sigma_3$ and $\iota$ is the generator of $W\langle(123)\rangle=\Z_2$. The set of relations is $\{\beta\alpha-\delta\gamma, \iota^2, \sigma^3, \tau^2, \tau\sigma\tau-\sigma^2,\sigma\tau\sigma-\tau\}$.

\begin{figure}[H]
\centering
\begin{tikzpicture}[scale=0.7]
\filldraw[black] (0,0) circle (2pt);
\node[above] at (0, 0.3){$\Sigma_3/\{e\}$};
\filldraw[black] (-3,3) circle (2pt) node[right] {$\Sigma_3/\langle (123)\rangle$};
\filldraw[black] (3,4) circle (2pt) node[right] {$\Sigma_3/\langle (12)\rangle$};
\filldraw[black] (0,6) circle (2pt) node[above] {$\Sigma_3/\Sigma_3$};
\node[below left] at (-1.5,1.5) {$\alpha$};
\node[above left] at (-1.5,4.5) {$\beta$};
\node[below right] at (1.5,2) {$\gamma$};
\node[above right] at (1.5,5) {$\delta$};
\node at (0,-1.5) {$\sigma$};
\node at (0,-2.5) {$\tau$};
\node at (-4.5,3) {$\iota$};
\draw[ultra thick, ->](-0.2,0.2)--(-2.8, 2.8);
\draw[ultra thick, ->](0.2,0.2)--(2.8, 3.8);
\draw[ultra thick, ->](-2.8,3.2)--(-0.2, 5.8);
\draw[ultra thick, ->](2.8,4.2)--(0.2, 5.8);
\draw [ultra thick, ->] (-0.3,0) arc(105:435:1);
\draw [ultra thick, ->] (-3,3.3) arc(15:345:1);
\draw [ultra thick, ->] (-0.45,0) arc(110:430:1.5);
\end{tikzpicture}
\caption{Reduced canonical category of $\Sigma_3$} 
\label{fig:S3}
\end{figure}
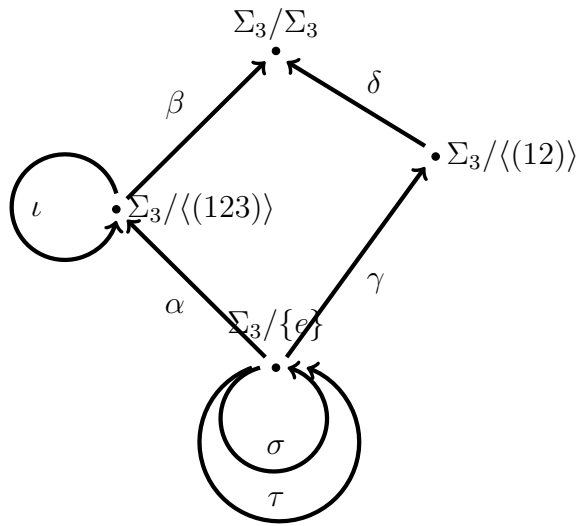
\end{example}

\section{Rational Bredon coefficient systems}
In order to apply the universal coefficient spectral sequence, we need to find an injective resolution for a given coefficient system. However, it could be a difficult task for a general underlying ring $R$. It is not known to the author whether there is a general method to construct an injective resolution for any coefficient system over the integers. However, over $\Q$, in [\cite{Doman88}], Doman constructed the injective envelope for any Bredon coefficient system over the rationals.

\begin{definition}
For any subgroup $H<G$, let $V_H$ be a left $\Q(WH)$-module. Define a Bredon coefficient system $I(V_H)$ by 
$$\underline{I(V_H)}(G/K)= \Hom_{\Q(WH)}(\Q((G/K)^H), V_H).$$ 
For a $G$-map $f:G/K\rightarrow G/K'$, the map $\overline{f}: \Q((G/K)^H)\rightarrow \Q((G/K')^H)$ is induced by $f$. Then 
$$\underline{I(V_H)}(f): \underline{I(V_H)}(G/K')\rightarrow \underline{I(V_H)}(G/K)$$ is defined by $\underline{I(V_H)}(f)(g)=g\circ \overline{f}$ where $g\in I(V_H)(G/K')$. 
\end{definition}

\begin{lemma}[\cite{Doman88}]
\label{injecitveness}
The coefficient system $I(V_H)$ is an injective object in the category $\cg^{\Q}$.
\end{lemma}

Moreover, the injective coefficient system $\underline{I(V_H)}$ could be used to construct an injective envelope for any Bredon coefficient system $\underline{M}$ over $\Q$. Let $V_{\{e\}}=\underline{M}(G/\{e\})$. In general, let $$V_H=\bigcap_{K<H}\Ker\underline{M}(f_{K,H})$$ where $f_{K,H}:G/K\rightarrow G/H$ is the projection in the canonical orbit category.

\begin{theorem}[\cite{Doman88}]
\label{injective envelope}
For any coefficient system $\underline{M}$ over $\Q$, the map $f:\underline{M}\rightarrow\oplus \underline{I(V_H)}$ is an injective envelope of $\underline{M}$, where the direct sum is over all conjugacy classes of $G$.
\end{theorem}

\begin{corollary}
The global (injective) dimension of $\cg^\Q$ is less than $L-1$ where $L$ is the largest length of subgroup chains in $G$.
\end{corollary}

Let $\underline{\Q}$ be the constant $\Q$ coefficient system.

\begin{lemma}
The constant coefficient system $\underline{\Q}$ is injective.
\end{lemma}

\begin{proof}
Take $H=\{e\}$, it follows that $WH=G$ and $(G/K)^H=G/K$. Then for trivial $WH$-module $V_H=\Q$, we have 
$$\Hom_{\Q(WH)}(\Q(G/K),V_H)=\Q.$$
Hence $\underline{I(V_H)}=\underline{\Q}$. 
\end{proof}

\begin{corollary}
The universal coefficient spectral sequence $$E_2^{p,q}=\Ext_{\cg^R}^{p,q}(\underline{H_*}(X), \underline{\Q})\Rightarrow H_G^n(X, \underline{\Q}),$$ collapses at the $E_2$-page. Therefore $$H^n_{G}(X,\underline{\Q})= \Hom_{\cg^{\Q}}(\underline{H_n}(X), \underline{\Q}).$$
\end{corollary}
\section{Ordered configuration spaces of Euclidean spaces}

In this section we will introduction the classical computation on the homology and cohomology of configuration spaces. For more details , the reader can refer to \cite{Cohen95}, \cite{CLM76}.

\begin{definition}
Given a space $M$, the ordered configuration space of $q$-tuples of distinct points in $M$ is $$\Conf(M,q)=\{(x_1,\cdots,x_q)\in M^q \mid x_i \neq x_j, \text{for } i\neq j\}.$$
\end{definition}

\begin{example}
When $M=\R^2$, $\Conf(R^2, q)$ is an Eilenberg-Mac Lane space of type $K(B_q, 1)$ where $B_q$ is the Artin's Braid groups on $q$ strands. 
\end{example}

We will mainly focus on the case $M=R^n$ in order to compute the Bredon cohomology of equivariant configuration spaces. The following theorem give the structure of integral cohomology ring of $\Conf(\R^n,q)$

\begin{theorem}[\cite{CLM76}]
The integral cohomology ring of configuration spaces $\Conf(\R^n,q)$ is given by the following data: 
\begin{enumerate}[label=\roman*)]
\item
For $n=1$, $\Conf(\R, q)$ is homotopy equivalent to the symmetric group of degree $q$ as a discrete space, or in other word, a finite discrete space with $q!$ points.
\item
For $n\geq 2$, the integral cohomology ring $H^*(\Conf(\R^n,q))$ is generated by elements $$A_{i,j}, 1\leq j<i\leq q.$$
where $A_{i,j}$ is of degree $n-1$ and the relations are given by 
\begin{enumerate}[label=\arabic*)]
\item
$A_{i,j}^2=0$.
\item
$A_{i,j}A_{i,k}=A_{k,j}(A_{i,k}-A_{i,j})$ for $j\leq k$; and 
\item
associativity and graded commutativity.
\end{enumerate}
\end{enumerate}
\end{theorem}

\begin{corollary}
\label{homology of configuration spaces}
\begin{enumerate}[label=\roman*)]
\item
Let $n\geq 2$, for $k=1,\cdots, q-1$, $H^{k(n-1)}(\Conf(\R^n, q))$, as an abelian group, has a basis $$\{A_{i_1,j_1}\cdots A_{i_k, j_k}\mid i_1<i_2<\cdots <i_k, \text{and } j_l<i_l \text{ for } l=1,2,\cdots, k\}.$$
\item
There is an abelian group isomorphism $$H^k(\Conf(\R^n, q))\cong H_k(\Conf(\R^n, q)).$$
for each $k$.
\end{enumerate}
\end{corollary}

\begin{example}
The rank of $H^{n-1}(\Conf(\R^n,q))$ is $\frac{q(q-1)}{2}$ and the rank of $H^{(q-1)(n-1)}(\Conf(\R^n,q))$ is $(q-1)!$.
\end{example}

We can compute the rank of a general degree using Poincar\'e series. 
\begin{proposition}[\cite{BCT89}]
\label{poincare series}
The Poincar\'e series of $H^{\ast}(\Conf(\R^n,q))$ is given by $$\prod_{m=1}^{q-1}(1+mt^{n-1}).$$
Hence rank of $H^{m(n-1)}(\Conf(\R^n,q))$ is $$\sum_{i_1<i_2<\cdots<i_m}i_1\cdot i_2\cdot \cdots \cdot i_m.$$
\end{proposition}

\begin{proposition}
\label{trivial inclusion}
For $m<n$, let $\iota: \R^m \rightarrow \R^n$ be a linear embedding. It induces a natural embedding of configuration spaces $$\overline{\iota}: \Conf(\R^m, q)\rightarrow \Conf(R^n,q).$$ Moreover the induced map on homology $$\overline{\iota}_*: H_*(\Conf(\R^m, q))\rightarrow H_*(\Conf(R^n,q))$$ is trivial.
\end{proposition}
\begin{proof}
It is directly from the following commutative diagram. 
\begin{figure}[H]
\centering
\begin{tikzpicture}[scale=1]
\node[] at (0,0) {$H_*(\Conf(R^m,q))$};
\node[] at (0,2) {$H_*(\Conf(R^n,q))$};
\node[] at (6,0) {$H_*(S^{m-1})$};
\node[] at (6,2) {$H_*(S^{n-1})$};
\node[above] at (3.5,2) {$A_{i,j}$};
\node[right] at (6,1) {$0$};
\node[above] at (3.5,0) {$A_{i,j}$};
\node[left] at (0,1) {};
\draw[ultra thick, <-](0,1.7)--(0,0.3);
\draw[ultra thick, <-](6,1.7)--(6,0.3);
\draw[ultra thick, ->](2,0)--(5,0);
\draw[ultra thick, ->](2,2)--(5,2);
\end{tikzpicture}
\end{figure}
The two horizontal maps are surjections given by cohomology class $A_{i,j}$ for any $1\leq j<i\leq q$.
  
\end{proof}
\section{Fixed point sets of equivariant configuration spaces}

\begin{definition}
Let $G$ be a finite group, for an $\R$-linear representation $V$, the equivariant configuration space of $V$ is the space $\Conf(V,q)$ with the diagonal $G$-action.
\end{definition}

\begin{lemma}
\label{fixed point set of ecs}
For any subgroup $H<G$, the H-fixed point $\Conf(V,q)^H$ is $\Conf(V^H,q)$.
\end{lemma}

\begin{proof}
For any $h\in H$ and any point $(x_1,\cdots,x_q)\in \Conf(V,q)^H$,
$$(x_1,\cdots,x_q)=h.(x_1,\cdots,x_q)=(h.x_1,\cdots,h.x_q).$$ Hence $h.x_i=x_i$ for each $i$. Then $x_i\in V^H$.
\end{proof}

Since $V$ is a linear representation, the fixed point set $V^H$ is a linear subspace of $V$. Moreover, if given $K$ and $H$ are two subgroups of $G$ such that $K<H<G$. Then the the embedding of fixed point sets $\Conf(V,q)^H\hookrightarrow \Conf(V,q)^K$ is induced by the embedding $V^H\hookrightarrow V^K$. By Proposition \ref{trivial inclusion}, the embedding is trivial on homology as long as the embedding $V^H\hookrightarrow V^K$ is proper. 

If $V=\R[G]$ is the regular representation of $G$. The above argument leads to an important decomposition of $\underline{H_*}(\Conf(V,q))$, the homology Bredon coefficent system of $\Conf(V,q)$. To state the decompostion theorem, we need to introduce the following class of ``1-dimensional'' coefficient system.

\begin{definition}
The Bredon coefficient system $\underline{\1_H}$ over $\Q$ is given by 
\[
\underline{\1_H}(G/K)=\left\{
\begin{array}{ll}
\Q,& \text{if $K$ congugates to $H$,}\\
0, & \text{otherwise. }
\end{array}\right.
\]
\end{definition}

\begin{theorem}
\label{decompostion of homology bcs}
For a finite group $G$, $V=\R[G]$ is the regular $\R$-linear representation of $G$. We have the following decomposition of the rational homology Bredon coefficent system of $\Conf(V,q)$,
\begin{enumerate}[label=\roman*)]
\item
For $n >0$, $$\underline{H_n}(\Conf(V,q))=\bigoplus_{H<G}\beta_{H,n}\underline{\1_H}$$ where
$\beta_{H,n}$ is the $n$-th Betti number of $\Conf(V^H,q)$.
\item
For $n=0$, $$\underline{H_0}(\Conf(V,q))=\underline{\Q}\oplus(q!-1)\underline{\1_{\{e\}}}$$
where $\underline{\Q}$ is the constant $\Q$ coefficient system.
\end{enumerate}
\end{theorem}
\begin{proof}
For any subgroup $H$, $V^H=\R^{[G:H]}$. In addition, if $K$ is a proper subgroup of $H$, $V^H$ is a proper subspace of $V^K$. Hence when $n>0$, by Proposition \ref{trivial inclusion}, the connection maps of $\underline{H_k}(\Conf(V,q))$ are all trivial except those isomorphisms induced by conjugation. Therefore $\underline{H_k}(\Conf(V,q))$ could be written as the direct sum of $\1_H$'s. The case when $n=0$ is a direct computation.
\end{proof}

\begin{corollary}
If $V$ is finitely generated free representation, i.e. $V=(\R[G])^s$ for some $s>1$.
Then
\[
\underline{H_n}(\Conf(V,q))=\left\{
\begin{array}{ll}
\bigoplus_{H<G}\beta_{H,n}\underline{\1_H}, & \text{for } n>0\\
\underline{\Q}, & \text{for } n=0
\end{array}\right.
\]
\end{corollary}

\begin{example}
Let $G=D_8$, and $V=\R[G]$, the regular representation of $G$. For any subgroup $H$, $V^H=\R^{[G:H]}$. In this example we compute the homology Bredon coefficient system of $\Conf(V,3)$. To simplify the notation, we denote denote subgroup of $G$ and canonical orbit category the same as in Section \ref{Computation for D8}.
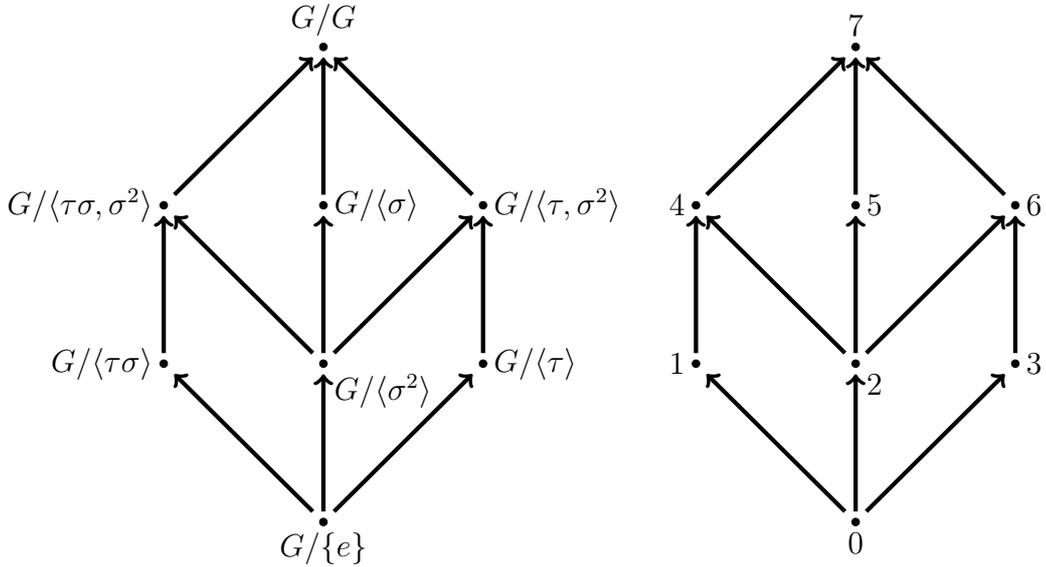
\begin{figure}[H]
\centering
\begin{tikzpicture}[scale=0.7]
\filldraw[black] (0,0) circle (2pt) node[below] {$G/\{e\}$};
\filldraw[black] (0,3) circle (2pt) node[below right] {$G/\langle\sigma^2\rangle$};
\filldraw[black] (-3,3) circle (2pt) node[left] {$G/\langle\tau\sigma\rangle$};
\filldraw[black] (3,3) circle (2pt) node[right] {$G/\langle\tau\rangle$};
\filldraw[black] (0,6) circle (2pt) node[right] {$G/\langle\sigma\rangle$};
\filldraw[black] (-3,6) circle (2pt) node[left] {$G/\langle\tau\sigma,\sigma^2\rangle$};
\filldraw[black] (3,6) circle (2pt) node[right] {$G/\langle\tau, \sigma^2\rangle$};
\filldraw[black] (0,9) circle (2pt) node[above] {$G/G$};
\draw[ultra thick, ->](0,0.2)--(0,2.8);
\draw[ultra thick, ->](-0.2,0.2)--(-2.8,2.8);
\draw[ultra thick, ->](0.2,0.2)--(2.8,2.8);
\draw[ultra thick, ->](0,3.2)--(0,5.8);
\draw[ultra thick, ->](-3,3.2)--(-3,5.8);
\draw[ultra thick, ->](3,3.2)--(3,5.8);
\draw[ultra thick, ->](-0.2, 3.2)--(-2.8, 5.8);
\draw[ultra thick, ->](0.2,3.2)--(2.8,5.8);
\draw[ultra thick, ->](-2.8,6.2)--(-0.2,8.8);
\draw[ultra thick, ->](2.8,6.2)--(0.2,8.8);
\draw[ultra thick, ->](0,6.2)--(0,8.8);
\filldraw[black] (10,0) circle (2pt) node[below] {0};
\filldraw[black] (10,3) circle (2pt) node[below right] {2};
\filldraw[black] (7,3) circle (2pt) node[left] {1};
\filldraw[black] (13,3) circle (2pt) node[right] {3};
\filldraw[black] (10,6) circle (2pt) node[right] {5};
\filldraw[black] (7,6) circle (2pt) node[left] {4};
\filldraw[black] (13,6) circle (2pt) node[right] {6};
\filldraw[black] (10,9) circle (2pt) node[above] {7};
\draw[ultra thick, ->](10,0.2)--(10,2.8);
\draw[ultra thick, ->](9.8,0.2)--(7.2,2.8);
\draw[ultra thick, ->](10.2,0.2)--(12.8,2.8);
\draw[ultra thick, ->](10,3.2)--(10,5.8);
\draw[ultra thick, ->](7,3.2)--(7,5.8);
\draw[ultra thick, ->](13,3.2)--(13,5.8);
\draw[ultra thick, ->](9.8, 3.2)--(7.2, 5.8);
\draw[ultra thick, ->](10.2,3.2)--(12.8,5.8);
\draw[ultra thick, ->](7.2,6.2)--(9.8,8.8);
\draw[ultra thick, ->](12.8,6.2)--(10.2,8.8);
\draw[ultra thick, ->](10,6.2)--(10,8.8);
\end{tikzpicture}
\caption{Repeated figure of $\mathscr{O}_{D_8}$}
\end{figure}
Then homology Bredon coefficient system is given by the Figure \ref{fig:HBCVD8}.
\begin{figure}[H]
\centering
\begin{tikzpicture}[scale=1]
\filldraw[black] (0,0) circle (2pt) node[below] {\small $H_*(\Conf(\R^8,3))$};
\filldraw[black] (0,3) circle (2pt) node[below right] {\small $H_*(\Conf(\R^4,3))$};
\filldraw[black] (-3,3) circle (2pt) node[left] {\small $H_*(\Conf(\R^4,3))$};
\filldraw[black] (3,3) circle (2pt) node[right] {\small $H_*(\Conf(\R^4,3))$};
\filldraw[black] (0,6) circle (2pt) node[right] {\small $H_*(\Conf(\R^2,3))$};
\filldraw[black] (-3,6) circle (2pt) node[left] {\small $H_*(\Conf(\R^2,3))$};
\filldraw[black] (3,6) circle (2pt) node[right] {\small $H_*(\Conf(\R^2,3))$};
\filldraw[black] (0,9) circle (2pt) node[above] {\small $H_*(\Conf(\R,3))$};
\draw[ultra thick, <-](0,0.2)--(0,2.8);
\draw[ultra thick, <-](-0.2,0.2)--(-2.8,2.8);
\draw[ultra thick, <-](0.2,0.2)--(2.8,2.8);
\draw[ultra thick, <-](0,3.2)--(0,5.8);
\draw[ultra thick, <-](-3,3.2)--(-3,5.8);
\draw[ultra thick, <-](3,3.2)--(3,5.8);
\draw[ultra thick, <-](-0.2, 3.2)--(-2.8, 5.8);
\draw[ultra thick, <-](0.2,3.2)--(2.8,5.8);
\draw[ultra thick, <-](-2.8,6.2)--(-0.2,8.8);
\draw[ultra thick, <-](2.8,6.2)--(0.2,8.8);
\draw[ultra thick, <-](0,6.2)--(0,8.8);
\end{tikzpicture}
\caption{Homology Bredon coefficient system of $\Conf(\R[G],3)$} 
\label{fig:HBCVD8}
\end{figure}
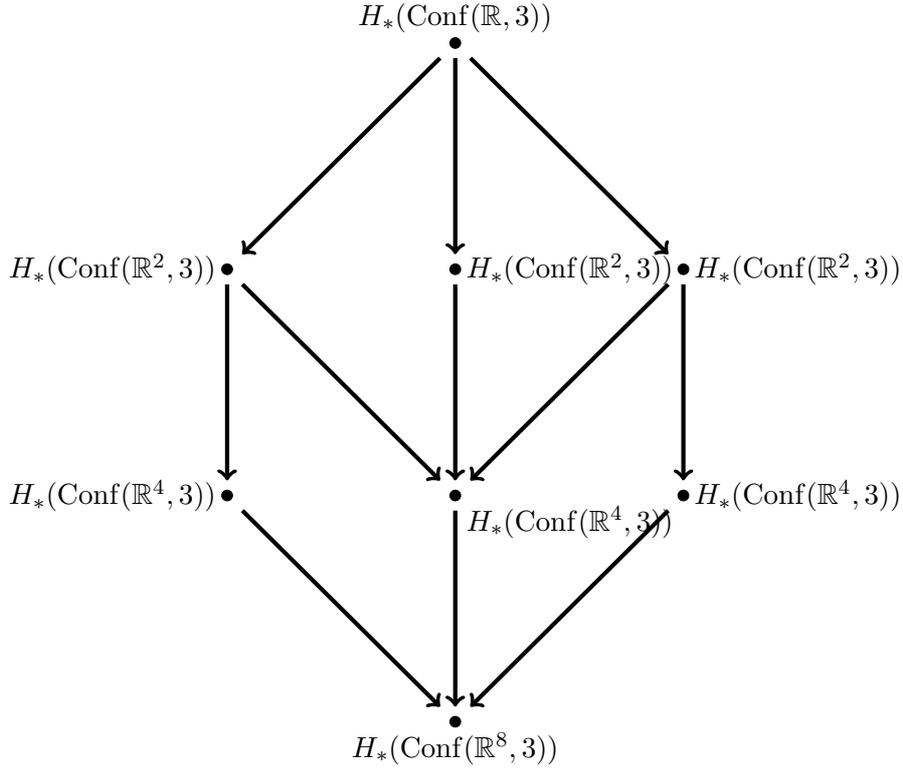
By Corollary \ref{homology of configuration spaces}, when $k>1$, 

\[
{H_n}(\Conf(\R^k,3))=\left\{
\begin{array}{ll}
\Q & \text{for } n=0\\
\Q^3, & \text{for } n=k-1\\
\Q^2, & \text{for } n=2(k-1)
\end{array}\right.
\]

Table \ref{Decomposition of Conf(V,3)} is the summary of the decomposition of the rational homology Bredon coefficient system of equivariant configuration space $\Conf(\R[G],3)$.

\begin{table}[H]
\centering
$\begin{array}{c|c}
\hline
\underline{H_n}(\Conf(\R[G],3))& \text{Decomposition}\\
\hline
n=0 & \underline{\Q}\oplus 5 \underline{\1_7}\\
n=1 & 3\underline{\1_4} \oplus 3\underline{\1_5} \oplus 3\underline{\1_6}\\
n=2 & 2\underline{\1_4} \oplus 2\underline{\1_5} \oplus 2\underline{\1_6}\\
n=3 & 3\underline{\1_1} \oplus 3\underline{\1_2} \oplus 3\underline{\1_3}\\
n=6 & 2\underline{\1_1} \oplus 2\underline{\1_2} \oplus 2\underline{\1_3}\\
n=7 & 3\underline{\1_0}\\
n=14 & 2\underline{\1_0}\\
\hline
\end{array}$
\caption{Decomposition of $\protect\underline{H_n}(\Conf(\R[G],3))$}
\label{Decomposition of Conf(V,3)}
\end{table}
\end{example}

\section{Rational Bredon cohomology of equivariant configuration spaces}
In this section we will apply universal coefficient spectral sequence to compute the rational Bredon cohomology of $\Conf(V,q)$. If $V$ is a free $G$-representation, we have given a direct sum decomposition of the homology coefficient system. Firstly, we have the following lemma.

\begin{lemma}
Given two Bredon coefficient system $\underline{M}$ and $\underline{N}$. If $\underline{M}=\oplus_k \underline{M_k}$, then 
$$\Hom_{\cg^\Q}(\underline{M},\underline{N})=\oplus_k\Hom_{\cg^\Q}(\underline{M_k},\underline{N})$$
\end{lemma}

Thus, the computation  $E_2$-page of the universal coefficient spectral sequence, we need to compute the $\Q$ dimension of $$\Hom_{\cg^\Q}(\underline{{1_H}},\underline{N}).$$

\begin{proposition}
\label{homorphism of 1dim}
Let $G$ be a finite group, H is a subgroup of $G$, for any rational coefficient system $\underline{N}\in \cg^\Q$, $$\Hom_{\cg^\Q}(\underline{{1_H}},\underline{N})\cong \underline{N}(G/H)\bigcap_{K<H}\Ker(i_K^H)$$ where $i_K^H: N(G/H)\rightarrow N(G/K)$ is the map induced by the natural projection $i:G/K\rightarrow G/H$.
\end{proposition}
\begin{proof}
Consider the following commutative diagram, 
\begin{figure}[H]
\centering
\begin{tikzpicture}[scale=1]
\node[] at (0,0) {$0$};
\node[] at (0,2) {$\Q$};
\node[] at (2,0) {$V$};
\node[] at (2,2) {$W$};
\node[above] at (1,2) {$h$};
\node[right] at (2,1) {$f$};
\node[above] at (1,0) {$ $};
\node[left] at (0,1) {$ $};
\draw[ultra thick, ->](0,1.7)--(0,0.3);
\draw[ultra thick, ->](2,1.7)--(2,0.3);
\draw[ultra thick, ->](0.3,0)--(1.7,0);
\draw[ultra thick, ->](0.3,2)--(1.7,2);
\end{tikzpicture}
\end{figure}
We have $f\circ h=0$. Namely, $h$ maps $\Q$ into the Kernel of $f$.  
\end{proof}

\subsection{Computation for 3-configuration spaces}
Similar to the polyhedral products case, in this subsection we consider the 1-dimensional coefficient system $\underline{\1_0}$ given by Figure \ref{fig:BC10}.
\begin{figure}[H]
\centering
\begin{tikzpicture}[scale=0.7]
\filldraw[black] (0,0) circle (2pt) node[below] {$\Q$};
\filldraw[black] (0,3) circle (2pt) node[below right] {$0$};
\filldraw[black] (-3,3) circle (2pt) node[left] {$0$};
\filldraw[black] (3,3) circle (2pt) node[right] {$0$};
\filldraw[black] (0,6) circle (2pt) node[right] {$0$};
\filldraw[black] (-3,6) circle (2pt) node[left] {$0$};
\filldraw[black] (3,6) circle (2pt) node[right] {$0$};
\filldraw[black] (0,9) circle (2pt) node[above] {$0$};
\draw[ultra thick, <-](0,0.2)--(0,2.8);
\draw[ultra thick, <-](-0.2,0.2)--(-2.8,2.8);
\draw[ultra thick, <-](0.2,0.2)--(2.8,2.8);
\draw[ultra thick, <-](0,3.2)--(0,5.8);
\draw[ultra thick, <-](-3,3.2)--(-3,5.8);
\draw[ultra thick, <-](3,3.2)--(3,5.8);
\draw[ultra thick, <-](-0.2, 3.2)--(-2.8, 5.8);
\draw[ultra thick, <-](0.2,3.2)--(2.8,5.8);
\draw[ultra thick, <-](-2.8,6.2)--(-0.2,8.8);
\draw[ultra thick, <-](2.8,6.2)--(0.2,8.8);
\draw[ultra thick, <-](0,6.2)--(0,8.8);
\end{tikzpicture}
\caption{$\protect\underline{\1_0}$} 
\label{fig:BC10}
\end{figure}

By Theorem \ref{injective envelope}, $\underline{\1_0}$ admits an injective resolution 
$$0\rightarrow \underline{\1_0}\rightarrow \underline{\Q}\rightarrow \underline{I_1} \rightarrow \underline{I_2}\rightarrow 0$$ 
The two new injectives $\underline{I_1}$ and $\underline{I_2}$ is given by Figure \ref{fig:I1BC10} and Figure \ref{fig:I2BC10}.

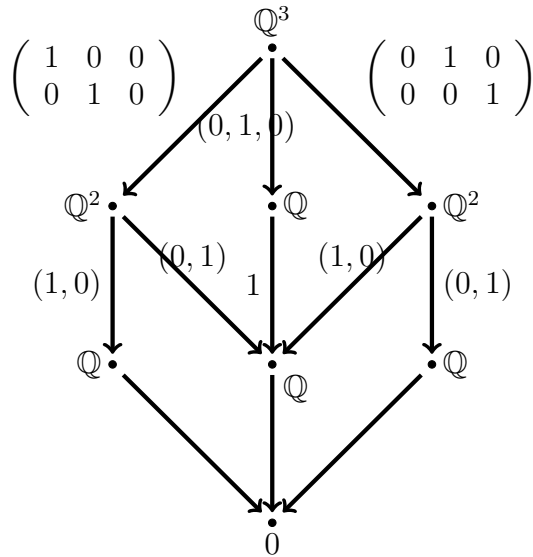
\begin{figure}[H]
\centering
\begin{tikzpicture}[scale=0.7]
\filldraw[black] (0,0) circle (2pt) node[below] {$0$};
\filldraw[black] (0,3) circle (2pt) node[below right] {$\Q$};
\filldraw[black] (-3,3) circle (2pt) node[left] {$\Q$};
\filldraw[black] (3,3) circle (2pt) node[right] {$\Q$};
\filldraw[black] (0,6) circle (2pt) node[right] {$\Q$};
\filldraw[black] (-3,6) circle (2pt) node[left] {$\Q^2$};
\filldraw[black] (3,6) circle (2pt) node[right] {$\Q^2$};
\filldraw[black] (0,9) circle (2pt) node[above] {$\Q^3$};
\draw[ultra thick, <-](0,0.2)--(0,2.8);
\draw[ultra thick, <-](-0.2,0.2)--(-2.8,2.8);
\draw[ultra thick, <-](0.2,0.2)--(2.8,2.8);
\draw[ultra thick, <-](0,3.2)--(0,5.8);
\draw[ultra thick, <-](-3,3.2)--(-3,5.8);
\draw[ultra thick, <-](3,3.2)--(3,5.8);
\draw[ultra thick, <-](-0.2, 3.2)--(-2.8, 5.8);
\draw[ultra thick, <-](0.2,3.2)--(2.8,5.8);
\draw[ultra thick, <-](-2.8,6.2)--(-0.2,8.8);
\draw[ultra thick, <-](2.8,6.2)--(0.2,8.8);
\draw[ultra thick, <-](0,6.2)--(0,8.8);
\node[above left] at (-1.5,7.5) {$\left(
\begin{array}{ccc}
1 &0 & 0\\
0 &1 & 0
\end{array}\right)$};
\node[above right] at (1.5,7.5) {$\left(
\begin{array}{ccc}
0 &1 & 0\\
0 &0 & 1
\end{array}\right)$};
\node[left] at (-3,4.5) {$(1,0)$};
\node[right] at (3,4.5) {$(0,1)$};
\node[above] at (-1.5,4.5) {$(0,1)$};
\node[above] at (1.5,4.5) {$(1,0)$};
\node[left] at (0,4.5) {$1$};
\node[] at (-0.5,7.5) {$(0,1,0)$};
\end{tikzpicture}
\caption{$\protect\underline{I_1}$} 
\label{fig:I1BC10}
\end{figure}

\begin{figure}[H]
\centering
\begin{tikzpicture}[scale=0.7]
\filldraw[black] (0,0) circle (2pt) node[below] {$0$};
\filldraw[black] (0,3) circle (2pt) node[below right] {$0$};
\filldraw[black] (-3,3) circle (2pt) node[left] {$0$};
\filldraw[black] (3,3) circle (2pt) node[right] {$0$};
\filldraw[black] (0,6) circle (2pt) node[right] {$0$};
\filldraw[black] (-3,6) circle (2pt) node[left] {$\Q$};
\filldraw[black] (3,6) circle (2pt) node[right] {$\Q$};
\filldraw[black] (0,9) circle (2pt) node[above] {$\Q^2$};
\draw[ultra thick, <-](0,0.2)--(0,2.8);
\draw[ultra thick, <-](-0.2,0.2)--(-2.8,2.8);
\draw[ultra thick, <-](0.2,0.2)--(2.8,2.8);
\draw[ultra thick, <-](0,3.2)--(0,5.8);
\draw[ultra thick, <-](-3,3.2)--(-3,5.8);
\draw[ultra thick, <-](3,3.2)--(3,5.8);
\draw[ultra thick, <-](-0.2, 3.2)--(-2.8, 5.8);
\draw[ultra thick, <-](0.2,3.2)--(2.8,5.8);
\draw[ultra thick, <-](-2.8,6.2)--(-0.2,8.8);
\draw[ultra thick, <-](2.8,6.2)--(0.2,8.8);
\draw[ultra thick, <-](0,6.2)--(0,8.8);
\node[above left] at (-1.5,7.5) {$(1,0)$};
\node[above right] at (1.5,7.5) {$(0,1)$};
\end{tikzpicture}
\caption{$\protect\underline{I_2}$} 
\label{fig:I2BC10}
\end{figure}
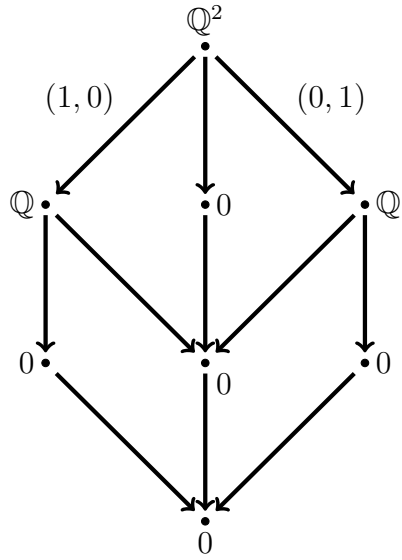

By Proposition \ref{homorphism of 1dim}, we have the following result.
\begin{enumerate}
\item
$\Hom_{\cg^Q}(\underline{\Q},\underline{\Q})=\Q$;
\item
$\Hom_{\cg^Q}(\underline{\1_0},\underline{\Q})=\Q$;
\item
$\Hom_{\cg^Q}(\underline{\1_k},\underline{\Q})=0$ for $k=1,2,\cdots, 7$;
\item
The map between $\Q$ and $\underline{I_1}$ is given in Figure \ref{D_8QI1}.

\begin{figure}[H]
\centering
\begin{tikzpicture}[scale=0.7]
\filldraw[black] (0,0) circle (2pt) node[below] {$\Q$};
\filldraw[black] (0,3) circle (2pt) node[below right] {$\Q$};
\filldraw[black] (-3,3) circle (2pt) node[left] {$\Q$};
\filldraw[black] (3,3) circle (2pt) node[right] {$\Q$};
\filldraw[black] (0,6) circle (2pt) node[right] {$\Q$};
\filldraw[black] (-3,6) circle (2pt) node[left] {$\Q$};
\filldraw[black] (3,6) circle (2pt) node[right] {$\Q$};
\filldraw[black] (0,9) circle (2pt) node[above] {$\Q$};
\draw[ultra thick, <-](0,0.2)--(0,2.8);
\draw[ultra thick, <-](-0.2,0.2)--(-2.8,2.8);
\draw[ultra thick, <-](0.2,0.2)--(2.8,2.8);
\draw[ultra thick, <-](0,3.2)--(0,5.8);
\draw[ultra thick, <-](-3,3.2)--(-3,5.8);
\draw[ultra thick, <-](3,3.2)--(3,5.8);
\draw[ultra thick, <-](-0.2, 3.2)--(-2.8, 5.8);
\draw[ultra thick, <-](0.2,3.2)--(2.8,5.8);
\draw[ultra thick, <-](-2.8,6.2)--(-0.2,8.8);
\draw[ultra thick, <-](2.8,6.2)--(0.2,8.8);
\draw[ultra thick, <-](0,6.2)--(0,8.8);
\filldraw[black] (10,0) circle (2pt) node[below] {$0$};
\filldraw[black] (10,3) circle (2pt) node[below right] {$\Q$};
\filldraw[black] (7,3) circle (2pt) node[left] {$\Q$};
\filldraw[black] (13,3) circle (2pt) node[right] {$\Q$};
\filldraw[black] (10,6) circle (2pt) node[right] {$\Q$};
\filldraw[black] (7,6) circle (2pt) node[left] {$\Q^2$};
\filldraw[black] (13,6) circle (2pt) node[right] {$\Q^2$};
\filldraw[black] (10,9) circle (2pt) node[above] {$\Q^3$};
\draw[ultra thick, <-](10,0.2)--(10,2.8);
\draw[ultra thick, <-](9.8,0.2)--(7.2,2.8);
\draw[ultra thick, <-](10.2,0.2)--(12.8,2.8);
\draw[ultra thick, <-](10,3.2)--(10,5.8);
\draw[ultra thick, <-](7,3.2)--(7,5.8);
\draw[ultra thick, <-](13,3.2)--(13,5.8);
\draw[ultra thick, <-](9.8, 3.2)--(7.2, 5.8);
\draw[ultra thick, <-](10.2,3.2)--(12.8,5.8);
\draw[ultra thick, <-](7.2,6.2)--(9.8,8.8);
\draw[ultra thick, <-](12.8,6.2)--(10.2,8.8);
\draw[ultra thick, <-](10,6.2)--(10,8.8);
\draw [ultra thick, ->,red] (0,2.8) to [out=-10,in=-170] (10,2.8);
\draw [ultra thick, ->,red] (0,5.8) to [out=-10,in=-170] (10,5.8);
\draw [ultra thick, ->,red] (0,9.2) to [out=10,in=170] (10,9.2);
\draw [ultra thick, ->,red] (-3,3.2) to [out=10,in=170] (7,3.2);
\draw [ultra thick, ->,red] (3,3.2) to [out=10,in=170] (13,3.2);
\draw [ultra thick, ->,red] (-3,6.2) to [out=10,in=170] (7,6.2);
\draw [ultra thick, ->,red] (3,6.2) to [out=10,in=170] (13,6.2);
\node[red] at (2,4.2) {$f_1$};
\node[red] at (5,2.7) {$f_2$};
\node[red] at (8,4.2) {$f_3$};
\node[red] at (1,7.2) {$f_4$};
\node[red] at (5,5.7) {$f_5$};
\node[red] at (9,7.2) {$f_6$};
\node[red] at (5,9.3) {$f_7$};
\end{tikzpicture}
\caption{$\Hom_{\cg^{\Q}}(\protect\underline{\Q}, \protect\underline{I_1})$} 
\label{D_8QI1}
\end{figure}
\begin{enumerate}
\item
We first assume $f_1=s$, $f_2=t$, $f_3=r$.
\item
Next we set $f_4=(a_1, a_2)^T$, then by the commutative diagram
\begin{figure}[H]
\centering
\begin{tikzpicture}[scale=1]
\node[] at (0,0) {$\Q$};
\node[] at (0,2) {$\Q$};
\node[] at (2,0) {$\Q$};
\node[] at (2,2) {$\Q^2$};
\node[above] at (1,2) {$(a_1,a_2)^T$};
\node[right] at (2,1) {$(1,0)$};
\node[above] at (1,0) {$f_1=s$};
\node[left] at (0,1) {$1$};
\draw[ultra thick, ->](0,1.7)--(0,0.3);
\draw[ultra thick, ->](2,1.7)--(2,0.3);
\draw[ultra thick, ->](0.3,0)--(1.7,0);
\draw[ultra thick, ->](0.3,2)--(1.7,2);
\end{tikzpicture}
\end{figure}
we have $a_1=s$. And by another commutative diagram
\begin{figure}[H]
\centering
\begin{tikzpicture}[scale=1]
\node[] at (0,0) {$\Q$};
\node[] at (0,2) {$\Q$};
\node[] at (2,0) {$\Q$};
\node[] at (2,2) {$\Q^2$};
\node[above] at (1,2) {$(a_1,a_2)^T$};
\node[right] at (2,1) {$(0,1)$};
\node[above] at (1,0) {$f_2=t$};
\node[left] at (0,1) {$1$};
\draw[ultra thick, ->](0,1.7)--(0,0.3);
\draw[ultra thick, ->](2,1.7)--(2,0.3);
\draw[ultra thick, ->](0.3,0)--(1.7,0);
\draw[ultra thick, ->](0.3,2)--(1.7,2);
\end{tikzpicture}
\end{figure}
we have $a_2=t$. Hence $f_4=(s,t)^T$.
\item
Next consider $f_5$, by the commutative diagram
\begin{figure}[H]
\centering
\begin{tikzpicture}[scale=1]
\node[] at (0,0) {$\Q$};
\node[] at (0,2) {$\Q$};
\node[] at (2,0) {$\Q$};
\node[] at (2,2) {$\Q$};
\node[above] at (1,2) {$f_5$};
\node[right] at (2,1) {$1$};
\node[above] at (1,0) {$f_2=t$};
\node[left] at (0,1) {$1$};
\draw[ultra thick, ->](0,1.7)--(0,0.3);
\draw[ultra thick, ->](2,1.7)--(2,0.3);
\draw[ultra thick, ->](0.3,0)--(1.7,0);
\draw[ultra thick, ->](0.3,2)--(1.7,2);
\end{tikzpicture}
\end{figure}
we have $f_5=t$.
\item
For $f_6$, assume that $$f_6=(b_1, b_2)^T.$$
From the commutative diagram
\begin{figure}[H]
\centering
\begin{tikzpicture}[scale=1]
\node[] at (0,0) {$\Q$};
\node[] at (0,2) {$\Q$};
\node[] at (2,0) {$\Q$};
\node[] at (2,2) {$\Q^2$};
\node[above] at (1,2) {$f_6$};
\node[right] at (2,1) {$(1,0)$};
\node[above] at (1,0) {$f_2=t$};
\node[left] at (0,1) {$1$};
\draw[ultra thick, ->](0,1.7)--(0,0.3);
\draw[ultra thick, ->](2,1.7)--(2,0.3);
\draw[ultra thick, ->](0.3,0)--(1.7,0);
\draw[ultra thick, ->](0.3,2)--(1.7,2);
\end{tikzpicture}
\end{figure}
we conclude that $b_1=t$. Similarly $b_2=r$. Hence
$f_6=(t,r)^T$.
\item
Finally for $f_7$, assume that $f_7=(c_1,c_2, c_3)^T.$
From the commutative diagram
\begin{figure}[H]
\centering
\begin{tikzpicture}[scale=1]
\node[] at (0,0) {$\Q$};
\node[] at (0,2) {$\Q$};
\node[] at (2,0) {$\Q^2$};
\node[] at (2,2) {$\Q^3$};
\node[above] at (1,2) {$f_7$};
\node[right] at (2,1) {$\left(
\begin{array}{ccc}
1 &0 & 0\\
0 &1 & 0
\end{array}\right)$};
\node[above] at (1,0) {$f_4$};
\node[left] at (0,1) {$1$};
\draw[ultra thick, ->](0,1.7)--(0,0.3);
\draw[ultra thick, ->](2,1.7)--(2,0.3);
\draw[ultra thick, ->](0.3,0)--(1.7,0);
\draw[ultra thick, ->](0.3,2)--(1.7,2);
\end{tikzpicture}
\end{figure}
we have $c_1=s$ and $c_2=t$. From the commutative diagram
\begin{figure}[H]
\centering
\begin{tikzpicture}[scale=1]
\node[] at (0,0) {$\Q$};
\node[] at (0,2) {$\Q$};
\node[] at (2,0) {$\Q^2$};
\node[] at (2,2) {$\Q^3$};
\node[above] at (1,2) {$f_7$};
\node[right] at (2,1) {$\left(
\begin{array}{ccc}
0&1 &0\\
0&0 &1
\end{array}\right)$};
\node[above] at (1,0) {$f_6$};
\node[left] at (0,1) {$1$};
\draw[ultra thick, ->](0,1.7)--(0,0.3);
\draw[ultra thick, ->](2,1.7)--(2,0.3);
\draw[ultra thick, ->](0.3,0)--(1.7,0);
\draw[ultra thick, ->](0.3,2)--(1.7,2);
\end{tikzpicture}
\end{figure}
we have $c_2=t, c_3=r$. Hence
$f_7=(s,t,r)^T$.
\item
In summary, there are three free variables in total and $$\Hom_{\cg^{\Q}}(\underline{\Q}, \underline{I_1})=\Q^3.$$
\end{enumerate}
\item
$\Hom_{\cg^Q}(\underline{\1_0},\underline{I_1})=0$;
\item
$\Hom_{\cg^Q}(\underline{\1_k},\underline{I_1})=\Q$ for $k=1,2,3$;
\item
$\Hom_{\cg^Q}(\underline{\1_k},\underline{I_1})=0$ for $k=4,5,6,7$;
\item
The map between $\underline{\Q}$ and $\underline{I_2}$ is given in Figure \ref{D_8QI2}.

\begin{figure}[h]
\centering
\begin{tikzpicture}[scale=0.7]
\filldraw[black] (0,0) circle (2pt) node[below] {$\Q$};
\filldraw[black] (0,3) circle (2pt) node[below right] {$\Q$};
\filldraw[black] (-3,3) circle (2pt) node[left] {$\Q$};
\filldraw[black] (3,3) circle (2pt) node[right] {$\Q$};
\filldraw[black] (0,6) circle (2pt) node[right] {$\Q$};
\filldraw[black] (-3,6) circle (2pt) node[left] {$\Q$};
\filldraw[black] (3,6) circle (2pt) node[right] {$\Q$};
\filldraw[black] (0,9) circle (2pt) node[above] {$\Q$};
\draw[ultra thick, <-](0,0.2)--(0,2.8);
\draw[ultra thick, <-](-0.2,0.2)--(-2.8,2.8);
\draw[ultra thick, <-](0.2,0.2)--(2.8,2.8);
\draw[ultra thick, <-](0,3.2)--(0,5.8);
\draw[ultra thick, <-](-3,3.2)--(-3,5.8);
\draw[ultra thick, <-](3,3.2)--(3,5.8);
\draw[ultra thick, <-](-0.2, 3.2)--(-2.8, 5.8);
\draw[ultra thick, <-](0.2,3.2)--(2.8,5.8);
\draw[ultra thick, <-](-2.8,6.2)--(-0.2,8.8);
\draw[ultra thick, <-](2.8,6.2)--(0.2,8.8);
\draw[ultra thick, <-](0,6.2)--(0,8.8);
\filldraw[black] (10,0) circle (2pt) node[below] {$0$};
\filldraw[black] (10,3) circle (2pt) node[below right] {$0$};
\filldraw[black] (7,3) circle (2pt) node[left] {$0$};
\filldraw[black] (13,3) circle (2pt) node[right] {$0$};
\filldraw[black] (10,6) circle (2pt) node[right] {$0$};
\filldraw[black] (7,6) circle (2pt) node[left] {$\Q$};
\filldraw[black] (13,6) circle (2pt) node[right] {$\Q$};
\filldraw[black] (10,9) circle (2pt) node[above] {$\Q^2$};
\draw[ultra thick, <-](10,0.2)--(10,2.8);
\draw[ultra thick, <-](9.8,0.2)--(7.2,2.8);
\draw[ultra thick, <-](10.2,0.2)--(12.8,2.8);
\draw[ultra thick, <-](10,3.2)--(10,5.8);
\draw[ultra thick, <-](7,3.2)--(7,5.8);
\draw[ultra thick, <-](13,3.2)--(13,5.8);
\draw[ultra thick, <-](9.8, 3.2)--(7.2, 5.8);
\draw[ultra thick, <-](10.2,3.2)--(12.8,5.8);
\draw[ultra thick, <-](7.2,6.2)--(9.8,8.8);
\draw[ultra thick, <-](12.8,6.2)--(10.2,8.8);
\draw[ultra thick, <-](10,6.2)--(10,8.8);
\draw [ultra thick, ->,red] (0,9.2) to [out=10,in=170] (10,9.2);
\draw [ultra thick, ->,red] (-3,6.2) to [out=10,in=170] (7,6.2);
\draw [ultra thick, ->,red] (3,6.2) to [out=10,in=170] (13,6.2);
\node[red] at (1,7.2) {$f_4$};
\node[red] at (9,7.2) {$f_6$};
\node[red] at (5,9.3) {$f_7$};
\end{tikzpicture}
\caption{$\Hom_{\cg^{\Q}}(\protect\underline{H_0}(Z(K;(D^1, S^0)), \protect\underline{I_2})$} 
\label{D_8QI2}
\end{figure}
\begin{enumerate}
\item
Firstly we assume that $f_4=r$ and $f_6=s$.
\item
From the following two commutative diagrams
\begin{figure}[H]
\centering
\begin{tikzpicture}[scale=1]
\node[] at (0,0) {$\Q$};
\node[] at (0,2) {$\Q$};
\node[] at (2,0) {$\Q$};
\node[] at (2,2) {$\Q^2$};
\node[above] at (1,2) {$f_7$};
\node[right] at (2,1) {$(1,0)$};
\node[above] at (1,0) {$r$};
\node[left] at (0,1) {$1$};
\draw[ultra thick, ->](0,1.7)--(0,0.3);
\draw[ultra thick, ->](2,1.7)--(2,0.3);
\draw[ultra thick, ->](0.3,0)--(1.7,0);
\draw[ultra thick, ->](0.3,2)--(1.7,2);
\end{tikzpicture}
\begin{tikzpicture}[scale=1]
\node[] at (0,0) {$\Q$};
\node[] at (0,2) {$\Q$};
\node[] at (2,0) {$\Q$};
\node[] at (2,2) {$\Q^2$};
\node[above] at (1,2) {$f_7$};
\node[right] at (2,1) {$(0,1)$};
\node[above] at (1,0) {$s$};
\node[left] at (0,1) {$1$};
\draw[ultra thick, ->](0,1.7)--(0,0.3);
\draw[ultra thick, ->](2,1.7)--(2,0.3);
\draw[ultra thick, ->](0.3,0)--(1.7,0);
\draw[ultra thick, ->](0.3,2)--(1.7,2);
\end{tikzpicture}
\end{figure}
We have $f_7=(s,t)^T$.
\item
In summary, $$\Hom_{\cg^{\Q}}(\underline{\Q}, \underline{I_2})=\Q^2.$$
\end{enumerate}
\item
$\Hom_{\cg^Q}(\underline{\1_k},\underline{I_2})=0$ for $k=0,1,2,3,5$;
\item
$\Hom_{\cg^Q}(\underline{\1_k},\underline{I_2})=\Q$ for $k=4,6$;
\item
$\Hom_{\cg^Q}(\underline{\1_k},\underline{I_2})=0$ for $k=7$;
\end{enumerate}

Combine with Table \ref{Decomposition of Conf(V,3)}, we have the $E_2$-page of the universal coefficient spectral sequence 
$$E_2^{p,q}=\Ext^q_{\cg^\Q}(\underline{H_p}(\Conf(V,3)),\underline{\1_0})$$
\begin{figure}[H]
\centering
\begin{tikzpicture}[scale=0.7]
\node[] at (1,1) {1};
\node[] at (8,1) {3};
\node[] at (15,1) {2};
\node[] at (1,2) {3};
\node[] at (4,2) {9};
\node[] at (7,2) {6};
\node[] at (1,3) {2};
\node[] at (2,3) {6};
\node[] at (3,3) {4};
\node[] at (1,-1) {0};
\node[] at (2,-1) {1};
\node[] at (3,-1) {2};
\node[] at (4,-1) {3};
\node[] at (5,-1) {4};
\node[] at (6,-1) {5};
\node[] at (7,-1) {6};
\node[] at (8,-1) {7};
\node[] at (9,-1) {8};
\node[] at (10,-1) {9};
\node[] at (11,-1) {10};
\node[] at (12,-1) {11};
\node[] at (13,-1) {12};
\node[] at (14,-1) {13};
\node[] at (15,-1) {14};
\draw[ultra thick, -](-1.5,0)--(16,0);
\draw[ultra thick, -](0,-1.5)--(0,4);
\end{tikzpicture}
\caption{$\Ext_{\cg^R}^{q}(\protect\underline{H_p}((\Conf(V,3)), \protect\underline{\1_0})$} 
\label{E_2pageConf3}
\end{figure}
There is one possible $d_2$-differential and the spectral sequence collapses at $E_3$-page. Hence we have the Bredon cohomology for $\Conf(\R[G],3)$ with coefficient $\underline{\1_0}$ for most of the degree except in degree 3 and 4. Further computation on the $d_2$-differential is needed to decide the rank on those degrees.

\begin{proposition}
The Bredon cohomology of $\Conf(\R[G],3)$ with coefficient system $\underline{\1_0}$ is given by the following table.
\begin{table}[H]
\centering
$\begin{array}{c|c}
\hline
& H_G^n(\Conf(\R[G],3),\1_0)\\
\hline
n=0 & \Q\\
n=1 & \Q^3\\
n=2 & \Q^2\\
n=3 & \Q^{6-k}\\
n=4 & \Q^{13-k}\\
n=7 & \Q^{9}\\
n=14 & \Q^{2}\\
\text{Otherwise} & 0\\
\hline
\end{array}$
\caption{Bredon cohomology $H_G^{\ast}(\Conf(\R[G],3),\protect\underline{\1_0})$}
\label{Bredon cohomology of Conf(R[G],3)}
\end{table}
for some integer $k\leq 6$.
\end{proposition}  

\subsection{Computation for 4-configuration spaces}
In the section we will compute the $E_2$-page of the universal coefficient spectral sequence for the equivariant configuration space $\Conf(\R[G],4)$ where $G=D_8$. For a general $q$-configuration spaces, the computation has no substantial difference. But there might be a few more notrivial $d_2$ differential and $d_3$-differential.

By Corollary \ref{homology of configuration spaces}, Proposition \ref{poincare series}, for $k\geq 1$, the rational homology of 4-configuration spaces is 

\[
{H_n}(\Conf(\R^k,4))=\left\{
\begin{array}{ll}
\Q & \text{for } n=0\\
\Q^6, & \text{for } n=k-1\\
\Q^{11}, & \text{for } n=2(k-1)\\
\Q^6, & \text{for } n=3(k-1)\\
\end{array}\right.
\]

Then by Theorem \ref{decompostion of homology bcs}, the decomposition of the rational Bredon coefficient system of $\Conf(\R[G],4)$ is given by the following table.

\begin{table}[H]
\centering
$\begin{array}{c|c}
\hline
\underline{H_n}(\Conf(\R[G],4))& \text{Decomposition}\\
\hline
n=0 & \underline{\Q}\oplus 23 \underline{\1_7}\\
n=1 & 6\underline{\1_4} \oplus 6\underline{\1_5} \oplus 6\underline{\1_6}\\
n=2 & 11\underline{\1_4} \oplus 11\underline{\1_5} \oplus 11\underline{\1_6}\\
n=3 &  6\underline{\1_1} \oplus 6\underline{\1_2} \oplus 6\underline{\1_3}\oplus 6\underline{\1_4} \oplus 6\underline{\1_5} \oplus 6\underline{\1_6}\\
n=6 & 11\underline{\1_1} \oplus 11\underline{\1_2} \oplus 11\underline{\1_3}\\
n=7 & 6\underline{\1_0}\\
n=9& 6\underline{\1_1} \oplus 6\underline{\1_2} \oplus 6\underline{\1_3}\\
n=14 & 11\underline{\1_0}\\
n=21 & 6\underline{\1_0}\\
\hline
\end{array}$
\caption{Decomposition of $\protect\underline{H_n}(\Conf(\R[G],4))$}
\label{Decomposition of Conf(R[G],4)}
\end{table}

Using the result on the dimension of related homomorphism set from previous subsection, we have the $E_2$-page of the universal coefficient spectral sequence.

$$E_2^{p,q}=\Ext^q_{\cg^\Q}(\underline{H_p}(\Conf(\R[G],4)),\underline{\1_0})$$
\begin{figure}[H]
\centering
\begin{tikzpicture}[scale=0.7,rotate=270,transform shape]
\node[] at (1,1) {1};
\node[] at (8,1) {6};
\node[] at (15,1) {11};
\node[] at (22,1) {6};
\node[] at (1,2) {3};
\node[] at (4,2) {18};
\node[] at (7,2) {33};
\node[] at (10,2) {18};
\node[] at (1,3) {2};
\node[] at (2,3) {12};
\node[] at (3,3) {22};
\node[] at (4,3) {12};
\node[] at (1,-1) {0};
\node[] at (2,-1) {1};
\node[] at (3,-1) {2};
\node[] at (4,-1) {3};
\node[] at (5,-1) {4};
\node[] at (6,-1) {5};
\node[] at (7,-1) {6};
\node[] at (8,-1) {7};
\node[] at (9,-1) {8};
\node[] at (10,-1) {9};
\node[] at (11,-1) {10};
\node[] at (12,-1) {11};
\node[] at (13,-1) {12};
\node[] at (14,-1) {13};
\node[] at (15,-1) {14};
\node[] at (16,-1) {15};
\node[] at (17,-1) {16};
\node[] at (18,-1) {17};
\node[] at (19,-1) {18};
\node[] at (20,-1) {19};
\node[] at (21,-1) {20};
\node[] at (22,-1) {21};
\draw[ultra thick, -](-1.5,0)--(23,0);
\draw[ultra thick, -](0,-1.5)--(0,4);
\end{tikzpicture}
\caption{$\Ext_{\cg^R}^{q}(\protect\underline{H_p}((\Conf(\R[G],4)), \protect\underline{\1_0})$} 
\label{E_2pageConf4}
\end{figure}
Similar to the 3-configuration space case, there is one possible $d_2$-differential and the spectral sequence collapses at $E_3$-page. Hence we have the Bredon cohomology for $\Conf(\R[G],4)$ with coefficient $\underline{\1_0}$ for most of the degree except in degree 3 and 4. Further computation on the $d_2$-differential is needed to decide the rank on those degrees.

\begin{proposition}
The Bredon cohomology of $\Conf(\R[G],4 )$ with coefficient system $\underline{\1_0}$ is given by the following table.
\begin{table}[H]
\centering
$\begin{array}{c|c}
\hline
& H_G^n(\Conf(\R[G],4),\1_0)\\
\hline
n=0 & \Q\\
n=1 & \Q^3\\
n=2 & \Q^2\\
n=3 & \Q^{12-k}\\
n=4 & \Q^{40-k}\\
n=5 & \Q^{12}\\
n=7 & \Q^{39}\\
n=10 & \Q^{18}\\
n=14 & \Q^{11}\\
n=21 & \Q^{6}\\
\text{Otherwise} & 0\\
\hline
\end{array}$
\caption{Bredon cohomology $H_G^{\ast}(\Conf(\R[G],4),\protect\underline{\1_0})$}
\label{Bredon cohomology of Conf(R[G],4)}
\end{table}
For some integer $k\leq 12$.
\end{proposition}  

Notice that since the constant coefficient system $\underline{\Q}$ is injective, the bottom rows in the universal coefficient spectral sequences in both Figure \ref{E_2pageConf3} and Figure \ref{E_2pageConf4} give the Bredon cohomology of $\Conf(\R[G],3)$ and $\Conf(\R[G],4)$ with constant $\Q$ coefficient system. Compared with the classic cohomology of the configuration spaces, 

\begin{corollary}
The Bredon cohomology of $\Conf(\R[G],3)$ and $\Conf(\R[G],4)$ with constant $\Q$ coefficient system is isomorphic to the classic cohomology the underlying configuration spaces. Namely,
$$H_G^{\ast}(\Conf(\R[V], k),\underline{\Q})\cong H^{\ast}(\Conf(\R^8),k)$$ for $k=3,4$.
\end{corollary}

\nocite{*}
\bibliographystyle{abbrv}

\bibliography{BCC}

\end{document}